\documentclass{lms}

\usepackage{amssymb, amsmath, mathrsfs}
\usepackage[dvipsnames]{xcolor}
\usepackage{enumerate}
\usepackage{tikz}
\usepackage{tikz-cd}
\usepackage{float}

\usepackage{subcaption}
\usepackage{graphicx}

\usepackage{natbib}
\usepackage{pgfplots}
\pgfplotsset{compat=newest, colormap/blackwhite}

\usepackage{url}
\usepackage{hyperref}

\newcommand{\N}{\mathbb{N}}
\newcommand{\R}{\mathbb{R}}

\newcommand{\Cr}{\mathscr{C}}
\newcommand{\Dr}{\mathscr{D}}

\newcommand{\CAD}{\text{CAD}}

\newcommand{\id}{\text{id}}
\newcommand{\sign}{\text{sign}}

\newcommand{\cornet}{\mathfrak{C}}
\usepackage{xargs}
\usepackage[colorinlistoftodos,prependcaption,textsize=footnotesize]{todonotes}
\newcommandx{\unsure}[2][1=]{\todo[linecolor=red,backgroundcolor=red!25,bordercolor=red,#1]{#2}}
\newcommandx{\change}[2][1=]{\todo[linecolor=blue,backgroundcolor=blue!25,bordercolor=blue,#1]{#2}}
\newcommandx{\info}[2][1=]{\todo[linecolor=OliveGreen,backgroundcolor=OliveGreen!25,bordercolor=OliveGreen,#1]{#2}}
\newcommandx{\improvement}[2][1=]{\todo[linecolor=Plum,backgroundcolor=Plum!25,bordercolor=Plum,#1]{#2}}
\newcommandx{\lucas}[2][1=]{\todo[linecolor=ForestGreen,backgroundcolor=ForestGreen!25,bordercolor=ForestGreen,#1]{#2}}
\newcommandx{\pierre}[2][1=]{\todo[linecolor=Blue,backgroundcolor=Blue!25,bordercolor=Blue,#1]{#2}}

\newtheorem{theorem}{Theorem}[section] 
\newtheorem{lemma}[theorem]{Lemma}     

\newtheorem{proposition}[theorem]{Proposition}

\newnumbered{assertion}{Assertion}    
\newnumbered{conjecture}{Conjecture}  
\newnumbered{definition}[theorem]{Definition}
\newnumbered{hypothesis}{Hypothesis}
\newnumbered{remark}[theorem]{Remark}
\newnumbered{note}{Note}
\newnumbered{observation}{Observation}
\newnumbered{problem}{Problem}
\newnumbered{question}{Question}
\newnumbered{algorithm}{Algorithm}
\newnumbered{example}{Example}
\newunnumbered{notation}{Notation} 

\simpleequations

\title[On some Exotic CADs and Cells]
 {On some Exotic Cylindrical Algebraic Decompositions and Cells} 

\author{Lucas Michel}

\classno{14P10 (primary), 68W30, 14Q30 (secondary)}

\extraline{The author is supported by the FNRS-DFG PDR Weaves
(SMT-ART) grant 40019202.}

\begin{document}
\maketitle

\begin{abstract}
Cylindrical Algebraic Decompositions (CADs) endowed with additional topological properties have found applications beyond their original logical setting, including algorithmic optimizations in CAD construction, robot motion planning, and the algorithmic study of the topology of semi-algebraic sets. 
In this paper, we construct explicit examples of CADs and CAD cells that refute several conjectures and open questions of J. H. Davenport, A. Locatelli, and G. K. Sankaran concerning these topological assumptions.

\end{abstract}

\section{Introduction}
\noindent Cylindrical Algebraic Decomposition (CAD) is a cornerstone tool in real algebraic geometry and underlies classical algorithms for Quantifier Elimination over real closed fields (QE) \cite{collins1975}. 
Beyond its foundational role in QE, CADs with strong topological properties have proved useful in various contexts. 
For instance, they play a role in algorithmic optimizations within CAD construction itself (see the clustering method  \cite{Arnon}, \cite{ARNON1988}), visualization of semi-algebraic sets, in applications such as the Piano Mover’s Problem (see \cite{piano}), and in the study of the topology of semi-algebraic sets (see for instance \cite{DLSregular} or \cite{lazard2010}).

In this setting, several natural topological assumptions on CADs and their cells (such as being closure-finite, well-bordered, or locally boundary connected, ...) frequently appear as convenient working hypotheses. 
These properties concern, in particular, CAD cells adjacencies, the behaviour of the boundaries of CAD cells and the way cells are arranged within the decomposition. 
A deeper analysis of these conditions is not only of theoretical interest but also of practical relevance: if a CAD (or a CAD cell) is known to satisfy certain topological properties, then some heavy computations can be safely avoided when designing or implementing algorithms.
However, despite their frequent use, a precise understanding of these topological properties and their interrelations remains incomplete in the literature.

From a topological point of view, a CAD cell of $\mathbb{R}^n$ is rather simple, since it is a topological ball (also called a cell), that is, a singleton or a set homeomorphic to an open ball of some dimension less than or equal to $n$ (see, for instance, Proposition 5.3 of \cite{Basu}). Note that this number is the dimension of the considered CAD cell. 
However, when considering its boundary, it is common to visualize a regular CAD cell (see below) in $\mathbb{R}^2$ or $\mathbb{R}^3$ and to assume that its closure behaves nicely. 
Similarly, one usually illustrates CADs of $\mathbb{R}^2$ or $\mathbb{R}^3$ with strong topological properties (for instance, a closure-finite and well-bordered CAD). 
In such examples, the closure of each cell can typically be expressed as a union of cells of the same CAD, and the dimension of the boundary decreases by exactly one compared to that of the original cell.
This intuition is misleading: the topology of the boundary of CAD cells, and their arrangement within their ambient CAD can be more intricate (especially when $n$ is large) as we will see below. 

In this paper, we construct explicit examples of CADs and CAD cells in $\R^4$ that challenge or refute some conjectures and open questions related to these topological aspects (see \cite{DLSregular} and the thesis \cite{locatelli}).
In Section \ref{sec:CFvsWB}, we construct a family of CADs in $\mathbb{R}^4$ that are closure-finite but not well-bordered. 
In Section \ref{sec:ExoticCells}, we construct exotic examples of CAD cells in $\mathbb{R}^4$ that exhibit topological phenomena that do not appear for CAD cells in lower-dimensional Euclidean spaces.

\subsection{Notation and framework}

For every positive integer $n \in \mathbb{N}^*$, we consider CADs of $\mathbb{R}^n$ with respect to a fixed variable ordering. Since the cells of such CADs are defined inductively from cells of CADs of $\mathbb{R}^k$ ($k \leq n$) which are indexed by $k$-tuples, it is convenient to adopt the tuple notation of \cite{miniCAD} and use the $\odot$ shorthand of \cite{binyamini2019complex} to deal easily with sectors and sections. 

More precisely, we identify any $k$-tuple $I=(i_1,\ldots,i_k)\in\mathbb{N}^k$ with the corresponding word $i_1\ldots i_k$. We say that $I$ is odd (resp. even) if $i_k$ is odd (resp. even). We denote by $\varepsilon$ the empty tuple, which corresponds to the empty word. For $j\in \mathbb{N}$, we denote by $I:j$ the $k+1$ tuple $(i_1,\ldots,i_k,j)$. For every subset $S \subseteq \R^n$, we denote (following the notation of \cite{bochnaketal1998}) by $\mathcal{S}^0(S)$ the ring of continuous semi-algebraic functions from $S$ to $\R$.
For $f \in \mathcal{S}^0(S)$, the section $S \odot \{f\}$ with base $S$ and bound $f$ is simply the graph of $f$, that is
    $$S \odot \{f\} = \{(\textbf{x},y)\in\R^{n+1} \; | \; \textbf{x} \in S,  y = f(\textbf{x})\}.$$
For $l \in \mathcal{S}^0(S) \cup \{-\infty\}$ and $u \in \mathcal{S}^0(S) \cup \{+ \infty\}$ such that $l < u$ on $S$, the sector $S \odot (l,u)$ with base $S$, lower bound $l$ and upper bound $u$, is the set
    $$S \odot (l,u) = \{(\textbf{x},y)\in\R^{n+1} \; | \; \textbf{x} \in S, l(\textbf{x}) < y < u(\textbf{x})\}.$$ 
    
Similarly to \cite{DLSregular}, we denote by $\overline{S}$ the closure of $S$ in the Euclidean topology, and by $\partial S = \overline{S} \setminus S$ its (cell) boundary. 
Note that this should not be confused with the usual topological frontier of $S$, which is defined as the closure of $S$ minus the interior of $S$.  
Recall that $S$ is locally boundary connected if, for every $p \in \partial S$, there exists $\delta > 0$ such that, for all $\varepsilon \in (0, \delta)$, the intersection $\mathbb{B}(p, \varepsilon) \cap S$ is connected, where $\mathbb{B}(p, \varepsilon)$ denotes the open ball in $\R^n$ of radius $\varepsilon$ and centred at~$p$.

\begin{definition}\label{def:cad}
    A cylindrical algebraic decomposition ($\CAD$) of $\mathbb{R}^n$ is a sequence $\mathscr{C} = (\mathscr{C}_1,\ldots,\mathscr{C}_n)$ such that  
    for all $k \in \{1,\ldots,n\}$, the set  $$\mathscr{C}_k= \left\{C_{i_1 \cdots i_k} |  \forall j \in \{1,\ldots,k\}, i_j \in \{1,\ldots,2u_{i_1\cdots i_{j-1}} +1\}\right\}$$
     is a finite semi-algebraic partition of $\mathbb{R}^k$  defined inductively by the following data:
    \begin{enumerate}
        \item there exist a natural number $u_\varepsilon \in \mathbb{N}$ and real algebraic numbers $\xi_{2} < \xi_{4} < \cdots <\xi_{2u_\varepsilon}$ (possibly none if $u_\varepsilon = 0$)
        that define exactly all cells of $\mathscr{C}_{1}$ by
        \begin{align*}
            C_{2j} &= \{\xi_{2j}\},\quad &(1 \leq j \leq u_\varepsilon)\\
             C_{2j+1}&= (\xi_{2j}, \xi_{2(j+1)}),\quad &(0 \leq j \leq u_\varepsilon)
        \end{align*}
         with the convention that $\xi_{0} = -\infty$ and $\xi_{2u_\varepsilon+ 2} = +\infty$;
        \item for each cell $C_I \in \mathscr{C}_k$ ($k < n$), there exist a natural number $u_I \in \mathbb{N}$ and $\xi_{I: 2 },\xi_{I:4},\ldots,\xi_{I:2 u_I} \in \mathcal{S}^0(C_I)$ (possibly none if $u_I = 0$) with
         $\xi_{I: 2 } < \xi_{I:4} < \ldots <\xi_{I:2 u_I}$ on $C_I$, that define exactly all cells of $\mathscr{C}_{k+1}$ by
        {\small\begin{align*}
             C_{I:2j} &= C_I \odot \{\xi_{I:2j}\} , \quad &(1 \leq j \leq u_I)\\
             C_{I:2j+1} &= C_I \odot (\xi_{I:2j},\xi_{I:2(j+1)}),\quad &(0 \leq j \leq u_I)
         \end{align*}}
         with the convention that $\xi_{I:0} = -\infty$ and $\xi_{I:2u_I 
+ 2} = +\infty$.
    \end{enumerate}
 We say that the element $C_{I}$ of $\mathscr{C}_k$ is a CAD cell of index $I$.  
\end{definition}

For all $k \in \{1, \ldots, n\}$, we denote by $\pi_k : \R^n \to \R^{k}$ the projection onto the first $k$ coordinates, defined by $\pi_k(x_1, \ldots, x_n) = (x_1, \ldots, x_k)$ , and use the same notation for its natural successive extensions to subsets of $\R^n$ and to families of subsets of $\R^n$.
In particular, for every $\CAD$ $\Cr$ of $\R^n$, we have $\pi_k(\Cr_n) = \Cr_k$ for every $ k\in \{1, \ldots, n\}$. For this reason, we usually identify $\Cr$ with $\Cr_n$ when the context is clear.

In general, one can refer to a CAD cell $C$ of $\R^n$ without referring to an ambient CAD, nor its index. If $n = 1$, then $C$ is either a singleton, or an open interval (with real-algebraic bounds). If $n > 1$, then $\pi_{n-1}(C)$ is a CAD cell of $\R^{n-1}$, and $C$ is a section or a sector with base $\pi_{n-1}(C)$ as defined above.

\section{Closure-finiteness and well-borderedness of CADs}\label{sec:CFvsWB}

In this section, we provide (via Theorem \ref{thrm:CFnotWB}) a negative answer to Question 3.1.2 of \cite{DLSregular}. 
We begin by recalling the relevant definitions and briefly reviewing the current state of the art.
\begin{definition} Let $\Dr$ be a finite partition of $\R^n$ such that each of its elements is a topological ball (as defined above). We say that $\Dr$ is 
    \begin{itemize}
        \item closure finite (CF) if for all $D \in \Dr$, there exist $D_1, \ldots, D_K \in \Dr$ such that $\overline{D} = \bigcup_{k=1}^K D_k$.
        \item well-bordered (WB) if for all $D \in \Dr$ of dimension $d \in \N^*$, there exist $D_1, \ldots, D_K \in \Dr$ of dimension $d - 1$ such that $\partial D = \bigcup_{k=1}^K \overline{D_k}$.
    \end{itemize}
\end{definition}

Decompositions satisfying the first condition are also called boundary coherent \cite{lazard2010}, or said to satisfy the frontier condition \cite{baker}. 
The second property is called boundary smooth in \cite{lazard2010}. 
Here we adopt the terminology of \cite{DLSregular}.

Corollary 2.2.18 of \cite{locatelli} asserts that, in general, the well-borderedness property is stronger than the closure finiteness property, i.e. if $\Dr$ is WB, then $\Dr$ is CF. In very low dimension, namely when $n \in \{1,2\}$, Corollary 2.2.20  of \cite{locatelli} guarantees that these two properties are equivalent. Moreover, Proposition 3.5.1 of \cite{locatelli} shows that if $n = 3$ and if $\Dr$ is a CAD, then these two properties are again equivalent. Conjecture 7.2.1 of \cite{locatelli} asserts that this fact is still true if we drop the assumption on the ambient dimension $n$. Theorem \ref{thrm:CFnotWB} below provides a counterexample, showing that the conjecture is false as stated.

\begin{theorem}\label{thrm:CFnotWB} \normalfont 
    There exist an infinite number of CADs of $\R^4$ that are closure finite but not well-bordered.
\end{theorem}

The remaining of this section is devoted to the proof of this result that invalidates Conjecture 7.2.1 of \cite{locatelli}, and gives a negative answer to Question 3.1.2 of \cite{DLSregular}.
The main idea of the proof is to generalize an example given by D. Lazard (see Example 2.11 of \cite{lazard2010}) of a CAD in $\mathbb{R}^4$ that is not well-bordered, and to refine it appropriately so as to preserve this property while also ensuring closure-finiteness. 

The following auxiliary lemma will be applied at the end of the proof to analyse the closure of some CAD cells. 

\begin{lemma}\label{lemma:consecutive} \normalfont 
        Let $S$ be a semi-algebraic set of $\R^{n}$ and consider the sectors and the section
        \begin{align*}
            A = S \odot (f_1,f_2), B = S \odot \{f_2\}, C = S \odot (f_2,f_3).
        \end{align*}
        We have $\overline{B} \subseteq \overline{A} \cap \overline{C}$. Moreover, if the sector $A \cup B \cup C = S \odot (f_1, f_3)$ is locally boundary connected, then $\overline{B} = \overline{A} \cap \overline{C}.$
    \end{lemma}
    Note that this result fails without the locally boundary connectedness assumption (see Remark \ref{rem:weirdAssumption} below). 
    \begin{proof}
    We first show that $B \subseteq \overline{A} \cap \overline{C}$, which will directly imply the first desired inclusion. If $(\mathbf{x}, y) \in B$, we can explicitly construct two sequences in $A$ and $C$ that converge to $(\mathbf{x}, y)$. For instance, consider the sequences $\left((\mathbf{x}, y - 1/k)\right)_{k \geqslant N}$ and $\left((\mathbf{x}, y + 1/k)\right)_{k \geqslant N}$, for $N$ large enough so that all their terms lie in $A$ and $C$, respectively.

    We show the other inclusion with the additional hypothesis that the set $D = A \cup B \cup C$ is locally boundary connected . We consider $(\textbf{x}, y) \in \overline{A} \cap \overline{C}$ and obtain that $\textbf{x}\in \overline{S}$ by the continuity of the projection $\pi_n$. We distinguish two cases. 
    First, if $\textbf{x} \in S$, then $y \leqslant f_2(\textbf{x})$ and $y \geqslant f_2(\textbf{x})$ since $(\textbf{x}, y) \in \overline{A} \cap \overline{C}$. This shows that $(\textbf{x}, y) \in B$.
    In the other case, if $\textbf{x} \notin S$, then $\textbf{x} \in \partial S$, and hence $(\textbf{x},y) \in \partial D$. By assumption, there exists $\delta > 0$ such that for every $\varepsilon \in (0, \delta)$, the intersection $\mathbb{B}((\textbf{x},y),\varepsilon) \cap D$ is connected. Also, observe that $\mathbb{B}((\textbf{x},y),\varepsilon)$ is a neighbourhood of $(\textbf{x}, y)$, and hence $\mathbb{B}((\textbf{x},y),\varepsilon) \cap A \neq \emptyset$ and $\mathbb{B}((\textbf{x},y),\varepsilon) \cap C \neq \emptyset$.  We fix $N \in \N$ sufficiently large (such that $1/N < \delta$), and observe that for every $k \geqslant N$, the continuous function 
    $$\mathbb{B}((\textbf{x},y),1/k) \cap D \to \R : (\textbf{z}, z_n) \to z_n - f_2(\textbf{z})$$
    is strictly negative on $\mathbb{B}((\textbf{x},y),1/k) \cap A$, strictly positive on $\mathbb{B}((\textbf{x},y),1/k) \cap C$, and hence must vanish at some point $(\textbf{x}^k, y^k)$ of the connected set $\mathbb{B}((\textbf{x},y),1/k) \cap D$. By construction, we must have $y^k = f_2(\textbf{x}^k)$, which is equivalent to $(\textbf{x}^k, y^k) \in B$. This defines a sequence $\left(\left(\textbf{x}^k, y^k\right)\right)_{k \geqslant N}$ in $B$ that converges to $(\textbf{x}, y)$, showing that $(\textbf{x}, y) \in \overline{B}$.
\end{proof}

\begin{proof}[of Theorem \ref{thrm:CFnotWB}]
We consider the CAD $\Cr_3$ of $\R^3$ given by these seven CAD cells : 
 \begin{align*}
    \begin{split}
         C_{111} &= (-\infty, 0)\times \R^2,\\
        C_{211} &= \{0\} \times (-\infty, 0) \times \R, 
    \end{split}
    \begin{split}
        C_{221} &= \{0\}^2\times (-\infty, 0),\\
        C_{222} &= \{0\}^3,\\
        C_{223} &= \{0\}^2\times (0,\infty),
    \end{split}
    \quad
    \begin{split}
        C_{231} &= \{0\} \times (0,\infty) \times \R,\\
        C_{311} &= (0,\infty) \times \R^2.
    \end{split}
    \end{align*}
    To define a CAD in $\R^4$ that is not well-bordered, we slice the cylinder above $C_{311}$ using a single function~$f$. 
This function is chosen so that the three-dimensional section $C_{311} \odot \{f\}$ has a one-dimensional boundary, and so that the closures of the sectors lying below and above it remain sufficiently simple to be described explicitly.

One possible way to proceed is to consider the set $\mathcal{W}$ of pairs $(f,r)$, where $f \in \mathcal{S}^0(C_{311})$ and $r$ is a real algebraic number, such that the two following properties are satisfied:
\begin{enumerate}[(i)]
    \item The boundary $\partial (C_{311} \odot \{f\})$ coincides with the half-line $H_r \subset \mathbb{R}^4$ defined by 
    \[
        x_1 = x_2 = x_3 = 0, \quad x_4 \leqslant r.
    \] \label{item:W1}
    \item If $\left((x_1^k, x_2^k, x_3^k)\right)_{k \in \mathbb{N}^*}$ is a sequence in~$C_{311}$ converging to a point $(x_1, x_2, x_3) \notin C_{311}$ with $x_2 \neq 0$ or $x_3 \neq 0$, then 
    \[
        \lim_{k \to \infty} f(x_1^k, x_2^k, x_3^k) = -\infty.
    \] \label{item:W2}
\end{enumerate}
    Elementary computations show that $\mathcal{W}$ is non-empty and even infinite. 
For instance, if $g$ is the function defined on $C_{311}$ by
$g(x_1, x_2, x_3) = -(x_2^2 + x_3^2)/x_1$ then the pairs $(g + s, s)$ (with $s$ a real-algebraic number) all belong to $\mathcal{W}$. For every $w = (f,r) \in \mathcal{W}$, we construct a CAD $\Cr^w_4$ of $\R^4$ from $\Cr_3$
by slicing the cylinder above $C_{222}$ by $r$ and the cylinder above $C_{311}$ by $f$. When there is no ambiguity, we will omit the reference to $w$ and write $\Cr_4$ in place of $\Cr_4^w$. The decomposition $\Cr_4$ of $\R^4$ contains exactly the following eleven CAD cells: 
    \begin{align*}
    \begin{split}
        C_{1111} & = C_{111} \times \R,\\
        C_{2111} & = C_{211} \times \R,\\
        C_{2211} & = C_{221} \times \R,\\
        C_{2231} & = C_{223} \times \R,\\
        C_{2311} & = C_{231} \times \R,\\
    \end{split}
    \begin{split}
    C_{2221} & = C_{222} \times (-\infty, r),\\
        C_{2222} & = C_{222} \times \{r\},\\
        C_{2223} & = C_{222} \times (r,\infty),
    \end{split}
    \quad \quad
    \begin{split}
        C_{3111} & = C_{311} \odot (-\infty, f),\\
        C_{3112} & = C_{311} \odot \{f\},\\
        C_{3113} & = C_{311} \odot (f,\infty). 
    \end{split}
    \end{align*}
To prove that $\Cr_4$ is closure-finite but not well-bordered, we analyse the closure of each of its eleven cells.

Since the closure of a product is the product of the corresponding closures, and since $\Cr_3$ is clearly closure-finite, it follows directly that the closure of every cell of $\Cr_4$, except for $\overline{C_{3111}}, \overline{C_{3112}}$, and $\overline{C_{3113}}$, is a union of cells of $\Cr_4$. We now show that this also holds for the three remaining cells, thereby proving that $\Cr_4$ is closure-finite.
Since $(f,r) \in \mathcal{W}$ and $H_r = C_{2221} \cup C_{2222}$, we immediately obtain 
\begin{equation}\label{eqn:C3112}
    \overline{C_{3112}} = C_{3112} \cup C_{2221} \cup C_{2222}.
\end{equation}
For the remaining cells, we show that
        \begin{align}
            \overline{C_{3111}} &=C_{3111}\cup  C_{3112} \cup C_{2221} \cup C_{2222}, \label{eqn:3111}\\
            \overline{C_{3113}} &=C_{3113} \cup C_{3112} \cup  C_{2221} \cup C_{2222} \cup C_{2223} \cup  C_{2111} \cup C_{2211} \cup C_{2231} \cup C_{2311}. \label{eqn:3113}
        \end{align}
The right-hand side of Equation \eqref{eqn:3111} is clearly a subset of the left-hand side, since Lemma~\ref{lemma:consecutive} and Equation \eqref{eqn:C3112} imply that $C_{3112} \cup C_{2221} \cup C_{2222} \subseteq \overline{C_{3111}}$.
Suppose, for the sake of contradiction, that there exists 
\[
\mathbf{x} = (x_1, x_2, x_3, x_4) \in \overline{C_{3111}} \setminus \left(C_{3111} \cup C_{3112} \cup C_{2221} \cup C_{2222}\right).
\]
Then there exists a sequence 
$
\bigl((x_1^k, x_2^k, x_3^k, x_4^k)\bigr)_{k \in \mathbb{N}}$ in $C_{3111}$ such that 
$\lim_{k \to \infty} (x_1^k, x_2^k, x_3^k, x_4^k) = \textbf{x}$. For each $k \in \mathbb{N}$, we have $x_1^k > 0$ and $x_4^k < f(x_1^k, x_2^k, x_3^k)$.  
Taking the limit as $k \to \infty$, it follows that $x_1 \geqslant 0$ and 
$x_4 \leqslant \lim_{k \to \infty} f(x_1^k, x_2^k, x_3^k)$.
If $x_1 > 0$, then $x_4 \leqslant f(x_1, x_2, x_3)$ by continuity, which contradicts the assumption that $\mathbf{x} \notin C_{3111} \cup C_{3112}$.  
If $x_1 = 0$, then we must have $x_2 = x_3 = 0$ (otherwise 
$x_4 \leqslant \lim_{k \to \infty} f(x_1^k, x_2^k, x_3^k) = -\infty$ 
by property \eqref{item:W2} of the definition of $\mathcal{W}$, which is absurd).  
Hence, $\mathbf{x} \in C_{2221} \cup C_{2222}$ by property \eqref{item:W1} of the definition of $\mathcal{W}$, which is again a contradiction. This shows that Equation \eqref{eqn:3111} is satisfied.
Finally, for Equation \eqref{eqn:3113}, we first observe that the union of the consecutive sectors and section $C_{3111}, C_{3112}, C_{3113}$ coincide with the half space of equation $x_1 > 0$, which is clearly locally boundary connected. The closure of this union is the union of the closures, but is also of equation $x_1 \geqslant 0$. The latter is precisely the disjoint union (denoted by $\sqcup$) of ten cells, specified in the right-hand side of the equality
\begin{equation}\label{eqn:long}
    \begin{split}
  \overline{C_{3111}} \cup \overline{C_{3112}} \cup \overline{C_{3113}} = &C_{3111} \sqcup C_{3113} \sqcup C_{3112} 
  \sqcup  C_{2221} \sqcup C_{2222} \\
  &\sqcup C_{2223} \sqcup  C_{2111} \sqcup C_{2211} \sqcup C_{2231} \sqcup C_{2311}.
\end{split}
\end{equation}
By Lemma \ref{lemma:consecutive}, we have  $\overline{C_{3111}}\cap\overline{C_{3113}} = \overline{C_{3112}}$, so the left-hand side also reads $\overline{C_{3111}}\cup\overline{C_{3113}}$, or even $(\overline{C_{3111}}\setminus\overline{C_{3112}} )\sqcup\overline{C_{3113}}$. 
Equation \eqref{eqn:3113} then follows directly from Equations \eqref{eqn:C3112} and \eqref{eqn:3111}.

Since the closure of every cell of $\Cr_4$ is a union of cell(s) of $\Cr_4$, the CAD $\Cr_4$ is closure finite. 
Moreover, $\Cr_4$ is not well-bordered, since by Equation \eqref{eqn:C3112}, the boundary of the three dimensional cell $C_{3112}$ coincides with the closure of the one dimensional cell $C_{2221}$. 
\end{proof}

\begin{remark}\label{rem:lazard}
    In connection with earlier work, D. Lazard studied in Example 2.11 of \cite{lazard2010} the roots in $x_4$ of the polynomial
    $x_1^2x_4^2 + (x_1^2+x_2^2)x_4 - (x_1^2+x_3^2)$
    for every $(x_1,x_2,x_3) \in C_{113}$ in order to exhibit a CAD of $\R^4$ that is not well-bordered. The negative root is precisely given by
    \[h(x_1, x_2, x_3) = \frac{-1}{2}\left(\left(1 + \left(\frac{x_2}{x_1}\right)^2\right) + \sqrt{ \left(1 + \left(\frac{x_2}{x_1}\right)^2\right)^2 + 4 \left(1 + \left(\frac{x_3}{x_1}\right)^2\right)}\right) .\]
    It is easily seen that the pair $w = \left(h, -(1+\sqrt{5})/{2}\right)$ belongs to $\mathcal{W}$. Similarly to Lazard’s construction, the resulting CAD $\Cr^w$ obtained above is not well-bordered because it contains $C_{311}\odot \{h\}$ as a section; however, in contrast, the CAD $\Cr^w$ is closure finite.
\end{remark}
\section{Some exotic CAD cells}\label{sec:ExoticCells}

The CAD cells in $\R^4$ we construct in this section will all be either sections or sectors built over the cornet cell (see Definition \ref{def:cornet}), itself a CAD cell in $\R^3$. Proposition \ref{lemma:equi-bases} presented below plays a key role to understand their topological properties at the boundary. In particular, it asserts that the properties under interest can in fact be studied on some corresponding sections and sectors in $\R^3$, which are not CAD cells, and whose bases are slit disks.

\subsection{Equiregularity, and the cornet cell}

As in \cite{DLSregular}, given the inclusions of topological spaces $X \subseteq X'$ and $Y \subseteq Y'$, a homeomorphism (of pairs) $\varphi : (X', X) \to (Y ', Y )$ is a homeomorphism $\varphi : X' \to Y '$ such that $\varphi(X) = Y$. 

\begin{definition}
    We say that two subsets $X \subset \R^n$ and $Y \subset \R^m$ are equiregular if there exists a homeomorphism of pairs $\varphi : (\overline{X}, X) \to (\overline{Y}, Y)$. A subset $S \subseteq \R^n$ is regular if it is a singleton, or if $S$ is equiregular with an open ball in $\R^d$ for some $d \in \N^*$.
\end{definition}

We present an example of two equiregular subsets. One is a CAD cell in $\mathbb{R}^3$, while the other, lying in $\mathbb{R}^2$, is not.

\begin{definition}\label{def:cornet}
    The cornet cell, denoted by $\cornet$, is the CAD cell $(0,2) \odot (g^-, g^+) \odot \{h\}$ where $g^\pm(x_1) = \pm\sqrt{1-(x_1-1)^2}$ and where $h(x_1,x_2) = (x_1^2+x_2^2)/(2x_1)$. We denote by $\mathbb{D}_{\text{s}}$ the slit disk, which is the open unit ball in $\R^2$ centred at the origin minus the radius $(-1,0]\times\{0\}$. 
\end{definition}
\begin{figure}[H]
    \centering
    \vspace{-1cm}
\begin{subfigure}{0.4\textwidth}
\hspace*{-1.3cm}\includegraphics[scale=0.95]{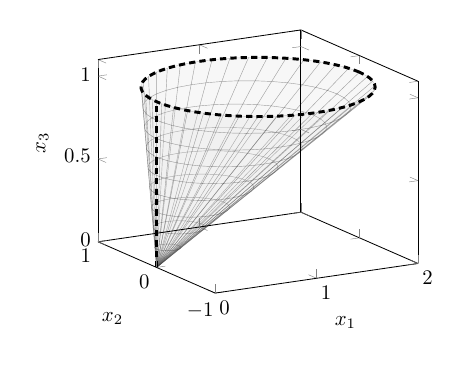}
\end{subfigure}~
\begin{subfigure}{0.4\textwidth}
\includegraphics[scale=0.9]{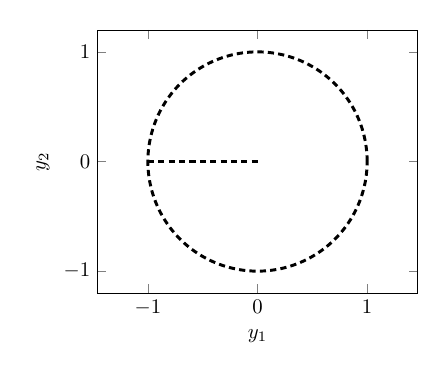}
\end{subfigure}
    \caption{The cornet cell $\cornet$ and the the slit disk $\mathbb{D}_{\text{s}}$. Their boundary are in dashed lines.}
    \label{fig:cornet-slit}
\end{figure}

It is readily observed that both the cornet cell and the slit disk (depicted in Figure \ref{fig:cornet-slit}) fail to be locally boundary connected. 
\begin{proposition}\label{prop:cornet-slitdisk}
    The cornet cell $\cornet$ is equiregular to the slit disk $\mathbb{D}_{\text{s}}$ via the semi-algebraic homeomorphisms of pairs
    \begin{align*}
      \varphi &: (\overline{\cornet}, \cornet) \to (\overline{\mathbb{D}_{\text{s}}},\mathbb{D}_{\text{s}}) : (x_1,x_2,x_3) \mapsto (x_1 - x_3, x_2),\\
      \psi &: (\overline{\mathbb{D}_{\text{s}}},\mathbb{D}_{\text{s}}) \to (\overline{\cornet}, \cornet) : (y_1,y_2) \mapsto \left(y_1 + \sqrt{y_1^2 + y_2^2}, y_2,  \sqrt{y_1^2 + y_2^2}\right),
    \end{align*}
    which are inverse to each other.
\end{proposition}
\begin{proof}
    We obtain the map $\psi$ as the composition of two elementary maps. First, we consider the continuous injection
    $f : \R^2 \to \R^3$ given by $f(y_1,y_2) = \left(y_1,y_2,\sqrt{y_1^2+y_2^2}\right)$ whose inverse is $\pi_2$, and the slit cone $C_{\text{s}} = f(\mathbb{D}_{\text{s}})$. We directly obtain that $f$ induces the homeomorphism of pairs $f|_{\overline{\mathbb{D}_{\text{s}}}} : (\overline{\mathbb{D}_{\text{s}}},\mathbb{D}_{\text{s}}) \to (\overline{C_{\text{s}}}, C_{\text{s}})$. Second, we consider the bijective affine map $g : \R^3 \to \R^3$ given by $g(z_1,z_2,z_3) = (z_1+z_3,z_2,z_3)$. It is straightforward to show that $g$ induces the homeomorphism of pairs $g|_{\overline{C_{\text{s}}}} : (\overline{C_{\text{s}}},C_{\text{s}}) \to (\overline{\cornet}, \cornet)$. 
    The conclusion follows since $\psi = g|_{\overline{C_{\text{s}}}} \circ f|_{\overline{\mathbb{D}_{\text{s}}}}$ and $\varphi = \psi^{-1}$.
\end{proof}

Roughly speaking, the following result asserts that, when the bases $X$ and $Y$ are equiregular, certain corresponding (via pullback) sectors and sections are also equiregular. 
This correspondence enables us to study sections and sectors over $X$ and $Y$ interchangeably. 
Furthermore, an analogous correspondence holds for specific fibres of interest.

\begin{proposition}\label{lemma:equi-bases} \normalfont 
    Let $X \subset \R^n$ and $Y \subset \R^m$ be equiregular sets, with the homeomorphism $\varphi : (\overline{X}, X) \to (\overline{Y}, Y)$.
    Every sector $Y \odot (l,u)$ and every section $Y \odot \{f\}$ are respectively equiregular to the the sector $X \odot (l \circ \varphi, u \circ \varphi)$ and to the section $X \odot \{f \circ \varphi\}$ via adequate restrictions of the homeomorphism $$\varphi \times \id_{\R} : (\overline{X} \times \R, X \times \R) \to (\overline{Y} \times \R, Y \times \R).$$ Furthermore, for every $\textbf{x} \in \overline{X}$, we have
    \begin{align*}
    \left(\varphi \times \id_{\R}\right)\left(\left(\pi_n|_{\overline{X \odot \{f \circ \varphi\}}}\right)^{-1}\left(\{\textbf{x}\}\right)\right) &= \left(\pi_m|_{\overline{Y \odot \{f\} }}\right)^{-1}\left(\{\varphi(\textbf{x})\}\right),\\
    \left(\varphi \times \id_{\R}\right)\left(\left(\pi_n|_{\overline{X \odot (l \circ \varphi, u \circ \varphi)}}\right)^{-1}(\{\textbf{x}\})\right) &= \left(\pi_m|_{\overline{Y \odot (l,u) }}\right)^{-1}\left(\{\varphi(\textbf{x})\}\right).    
    \end{align*}
\end{proposition}

\begin{proof}
    We give the proof for the sections only since the one for the sectors is analogous. 
    We show that $\varphi \times \id_\R$ induces a homeomorphism of pairs between $(\overline{X \odot \{f \circ \varphi\}}, X \odot \{f \circ \varphi\})$ and $(\overline{Y \odot \{f\}}, Y \odot \{f\})$.
    We first notice that $\varphi\times \id_\R$ is a continuous map from $\overline{X}\times \R$ to $\overline{Y}\times \R$, being the product of
    two continuous maps. It is a homeomorphism since its inverse $\varphi^{-1}\times  \id_\R$ is also continuous. It follows directly from the definitions that we have 
    $ (\varphi \times \id_\R) (X\odot \{f\circ \varphi\})\subseteq Y\odot \{f\}$. This inclusion is actually an equality since the same argument holds for  $\varphi^{-1}\times \id_\R$ and $Y\odot \{f\circ\varphi\circ \varphi^{-1}\}$. Now, $\overline{X\odot \{f\circ \varphi\}}$  is a subset of $\overline{X}\times \R$ since $\pi_n$ is continuous. Since $\varphi\times\id_\R$ is a homeomorphism, we thus have 
    $(\varphi \times \id_\R)(\overline{X\odot \{f\circ \varphi\}}) =\overline{Y\odot \{f\}}.$ 
    
    The remaining equalities follow immediately from the first part of the proof, and the fact that $\left(\pi_k|_A\right)^{-1}(\{\textbf{z}\}) = \left(\{\textbf{z}\} \times \R\right) \cap A$ for every $k \in \N^*, \textbf{z} \in \R^k, A \subseteq \R^{k+1}$. 
\end{proof}

\subsection{Applications}

We apply the results developed in the previous subsection to construct explicit CAD cells in $\R^4$ with some topological properties of interest. The proofs can readily be adapted to produce infinitely many distinct counterexamples. For the sake of presentation, we focus on some simple representative for each property.

In this subsection, we consider a semi-algebraic homeomorphism $\varphi$ between the pairs $(\overline{\cornet}, \cornet)$ and $(\overline{\mathbb{D}_{\text{s}}}, \mathbb{D}_{\text{s}})$ (see Proposition \ref{prop:cornet-slitdisk}).


\subsubsection{Closure of CAD cells and contractibility.}

In Theorem 3.18 of \cite{DLSregular}, it is shown that if $\Cr_n$ is a CAD of $\R^n$ whose projection $\Cr_{n-1}$ is a strong CAD of $\R^{n-1}$ (that is, a well-bordered CAD in which every cell is locally boundary connected), then the closure of any bounded cell $C$ of $\Cr_n$ is contractible.
The proof relies on Proposition 5.2 of \cite{lazard2010}, which requires the same assumption on $\Cr_{n-1}$ and states that for such a cell $C$, and for every $\textbf{x} \in \overline{\pi_{n-1}(C)}$, the fibre $(\pi_{n-1}|_{\overline{C}})^{-1}(\{\mathbf{x}\})$ is a closed segment.
It has been conjectured (see Conjectures 7.2.3 and 7.2.4 of \cite{locatelli}, as well as Question 3.19 of \cite{DLSregular}) that these results hold without the strongness assumption.
We disprove these conjectures by constructing explicit counterexamples in Proposition \ref{prop:fibre} and Proposition \ref{prop:closureContractible}. 

\begin{proposition}\label{prop:fibre}\normalfont
    There exist a CAD cell $C$ in $\R^4$ and a point $\textbf{x} \in \overline{\pi_3(C)}$ such that the fibre $(\pi_{3}|_{\overline{C}})^{-1}(\{\textbf{x}\})$ is a doubleton. In particular, this fibre is not a closed segment.
\end{proposition}
Note that we work in $\mathbb{R}^4$, rather than  a lower-dimensional Euclidean spaces, due to Corollary 3.3.34 of \cite{arnon-thesis}, which asserts that, in those settings, the fibre of interest is always a closed segment.

\begin{proof}
    We consider the continuous semi-algebraic function
    \begin{equation}\label{eqn:f-trousers}
        f : \mathbb{D}_{\text{s}} \to \R : (y_1,y_2) \mapsto \begin{cases}
        y_1 &\text{ if } y_1 < 0 \text{ and } y_2 <0,\\
        0 &\text{ otherwise}.
    \end{cases}
    \end{equation}
    We consider the CAD cell $C = \cornet \odot (f \circ \varphi)$ in $\R^4$ and the point $\textbf{x} = (0,0,1/2) \in \overline{\cornet}$. Proposition \ref{lemma:equi-bases} asserts that $(\pi|_{\overline{C}})^{-1}(\{\textbf{x}\})$ is homeomorphic to $(\pi|_{\overline{\mathbb{D}_{\text{s}} \odot \{f\}}})^{-1}(\{\varphi(\textbf{x}\}))$. We show that the latter contains exactly two distinct points, as depicted in Figure \ref{fig:trousers-fibre}. First, it is straightforward that we have $\varphi(\textbf{x}) = (-1/2,0)$ and
        \begin{align*}
            \{(-1/2,0,-1/2), (-1/2,0,0)\} \subseteq (\pi|_{\overline{\mathbb{D}_{\text{s}} \odot \{f\}}})^{-1}(\{\varphi(\textbf{x}\})).
        \end{align*}
    This inclusion in an equality since $\mathbb{D}_{\text{s}} \odot \{f\}$ is contained in the closed set of equation $y_3(y_3-y_1) = 0$, and so its closure is as well.
\end{proof}

\begin{figure}
    \centering
   \includegraphics[scale=1]{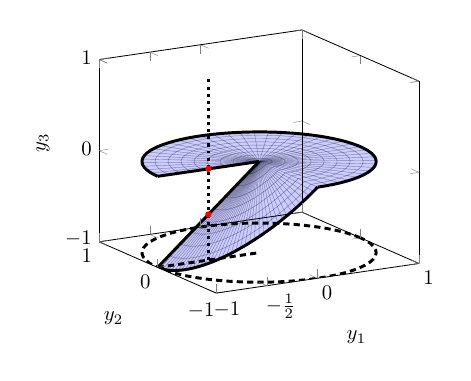}
    \caption{The the section $\mathbb{D}_{\text{s}} \odot \{f\}$ (blue), its boundary (thick black) and of the fibre $(\pi|_{\overline{\mathbb{D}_{\text{s}} \odot \{f\}}})^{-1}(\{\varphi(\textbf{x})\})$ (red dots) of Proposition \ref{prop:fibre}.}
    \label{fig:trousers-fibre}
\end{figure}

In the previous proof, the graph of $f$, which is the section $\mathbb{D}_{\text{s}} \odot \{f\}$, is reminiscent to the Trousers of \cite{miniCAD}, although it differs slightly. In particular, $\mathbb{D}_{\text{s}} \odot \{f\}$ is bounded, and no point of this section projects to the set of equation $(y_1<0) \land  (y_2 = 0)$ (the slot of the slit disk~$\mathbb{D}_{\text{s}}$).

\begin{remark}\label{rem:weirdAssumption}
    We provide a counterexample to the second part of Lemma \ref{lemma:consecutive} (in the case where the sector $A \cup B \cup C$ is not locally boundary connected) based on the map $f$ defined in Equation \eqref{eqn:f-trousers}. Consider $S = \mathbb{D}_{\text{s}}$, together with $f_1 = -\infty, f_2 = f$, and $f_3 = +\infty$. The point $\textbf{x} = (-1/2,0,-1/4)$ belongs to the intersection $\overline{A} \cap \overline{C}$, but not to $\overline{B}$.
    By equiregularity, we readily obtain a counter-example where $A,B,C$ and their union are in addition all CAD cells. For instance, take $S = \cornet$, $f_1 = -\infty$, $f_2 = f \circ \varphi$, $f_3 = + \infty$ and $\textbf{x} = \varphi^{-1}\times \id_{\R}(-1/2,0,-1/4)$.
\end{remark}

\begin{proposition}\label{prop:closureContractible}\normalfont
    There exists a bounded CAD cell in $\R^4$ whose closure is not simply connected (and hence not contractible).
\end{proposition}
\begin{proof}
We consider the semi-algebraic continuous map 
\begin{align*}
            f : \mathbb{D}_{\text{s}} \to \R : (y_1,y_2) \mapsto \begin{cases}
                \sign(y_2)\sqrt{\frac{1}{4}-(y_1+\frac{1}{2})^2} &\text{ if } -1 < y_1 < 0,\\
                0 &\text{ otherwise.}
            \end{cases}
        \end{align*}
whose graph $\mathbb{D}_{\text{s}} \odot \{f\}$ is depicted in Figure \ref{fig:conj724}. We consider the CAD cell $\cornet \odot (f \circ \varphi)$ and show that its closure is not simply connected. By Proposition \ref{lemma:equi-bases}, the set $\cornet \odot (f \circ \varphi)$ is equiregular to $\mathbb{D}_{\text{s}} \odot \{f\}$, so that $\overline{\cornet \odot (f \circ \varphi)}$ is homeomorphic to $\overline{\mathbb{D}_{\text{s}} \odot \{f\}}$. It is therefore sufficient to show that the latter is not simply connected in order to conclude that the former is not. To this end, we observe that the circle $A$ that lies in the plane $y_2 = 0$, centred at $(-1/2, 0,0)$ and of radius $1/2$ is a deformation retract of $\overline{\mathbb{D}_{\text{s}} \odot \{f\}}$, via the deformation retraction $$F : \overline{\mathbb{D}_{\text{s}} \odot \{f\}} \times [0,1] \to \overline{\mathbb{D}_{\text{s}} \odot \{f\}}$$ defined by 
\begin{align*}
    F\left((y_1,y_2,y_3),t\right)=\left(\frac{(1-t)(y_1+|y_1|)+(y_1-|y_1|)}{2}, (1-t)y_2, y_3\right).
\end{align*}
Note that the first component of $F$ is alternatively given by $(1-t)y_1$ if $y_1 > 0$ and $y_1$ otherwise. 
The conclusion follows from the fact that a deformation retraction is a homotopy equivalence, and that the circle $A$ is not simply connected.
\end{proof}

\begin{figure}
\begin{subfigure}{0.4\textwidth}
    \hspace*{-2.3cm}\includegraphics[scale = 1]{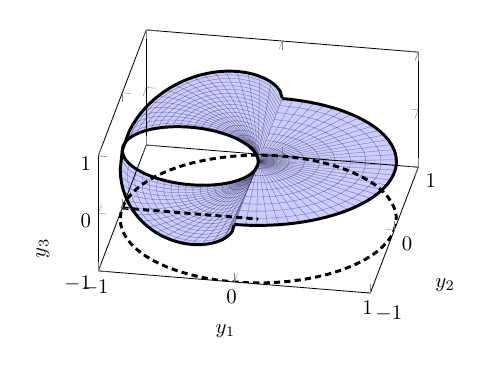}
    \end{subfigure}~\hspace{-0.9cm}~
    \begin{subfigure}{0.4\textwidth}
        \includegraphics[scale = 1]{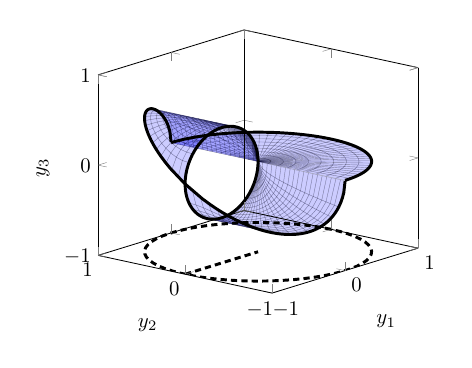}
    \end{subfigure}
        \caption{Two different views of the section $\mathbb{D}_{\text{s}} \odot \{f\}$ (blue) and its boundary (thick black) of Proposition \ref{prop:closureContractible}.}
        \label{fig:conj724}
\end{figure}

\subsubsection{Non-regular sectors with regular bounds.}
We provide a counterexample to Conjecture 7.2.2 of \cite{locatelli}, which asserts that, in general, a sector bounded above and below by regular sections is itself regular.
The motivation behind this conjecture lies in its potential to generalize Lemma 5 of \cite{piano} (see also Lemma 4.1.1 of \cite{locatelli}), and thus to extend the main theorem of \cite{DLSregular} (Theorem 3.11), which relies on that lemma.

\begin{proposition}\label{prop:conj722}
    There exists a CAD cell $\cornet \odot (l,u)$ of $\R^4$, which is a non-regular sector and such that the sections $\cornet \odot \{l\}$ and $\cornet \odot \{u\}$ are both regular.
\end{proposition}
\begin{proof}
    We first define the semi-algebraic continuous functions $l$ and $u$ on the cornet cell. We denote by $H \subset \R^2$ the open half-disk of equation $(z_1^2+z_2^2 < 1) \land (z_1 > 0)$ and consider the map 
    $$m : \overline{H} \to \R^3 :  (z_1,z_2) \mapsto \left(z_1^2-z_2^2,2z_1z_2,z_2+2\sqrt{z_1^2+z_2^2}\right).$$
    Note that the two first components are the real and imaginary part of the square of the complex number $z_1 + z_2 i$. This map $m$ is a continuous injection from a compact to a Hausdorff space, and hence a homeomorphism between $\overline{H}$ and its image. Denoting by $P = m(H)$, this map clearly induces a homeomorphism of pairs $(\overline{H},H) \to (\overline{P},P)$, i.e. $P$ and $H$ are equiregular.  Furthermore, the set $P$ is a section above $\mathbb{D}_{\text{s}}$ since the two first components of $m$ define a homeomorphism between $H$ and $\mathbb{D}_{\text{s}}$. More precisely, the algebraic formula to compute the principal square root of a complex number gives that $P = \mathbb{D}_{\text{s}} \odot \{\widetilde{u}\}$, where 
    $$\widetilde{u}(y_1,y_2) = \sign(y_2)\sqrt{(\sqrt{y_1^2 + y_2^2} - y_1)/2} + 2\sqrt[4]{y_1^2 + y_2^2}.$$
    We consider the strictly positive continuous maps $u = \widetilde{u} \circ \varphi$ defined on $\cornet$ and write $l = -u$.
     
     We show that the sections $\cornet \odot \{l\}, \cornet \odot \{u\}$ are regular, but the sector $\cornet \odot (l,u)$ is not. By Proposition \ref{lemma:equi-bases}, we obtain that these three CAD cells are respectively equiregular to the corresponding section $\mathbb{D}_{\text{s}} \odot \{-\widetilde{u}\},\mathbb{D}_{\text{s}} \odot \{\widetilde{u}\}$ or the corresponding sector $\mathbb{D}_{\text{s}} \odot (-\widetilde{u}, \widetilde{u})$ (see Figure \ref{fig:conj722}). To show that the former is regular or not, it is equivalent to show that the latter is.
    First, we already showed that the section $P = \mathbb{D}_{\text{s}} \odot \{\widetilde{u}\}$ is equiregular to the set $H$, which is obviously regular. 
    Second, for the sake of a contradiction, we assume that $\mathbb{D}_{\text{s}} \odot (-\widetilde{u}, \widetilde{u})$ is regular. In this case, the closure of this sector would homeomorphic to a closed unit ball in $\R^3$. However, the once-punctured set $S = \overline{\mathbb{D}_{\text{s}} \odot (-\widetilde{u}, \widetilde{u})} \setminus \{(0,0,0)\}$ is not simply connected (since the punctured disk of equation $(0<y_1^2+y_2^2<1) \land (y_3 = 0)$  is a deformation retract of $S$), whereas any once-punctured unit closed ball in $\R^3$ is. 
\end{proof}

\begin{figure}[H]
    \centering
    \includegraphics[scale=1.2]{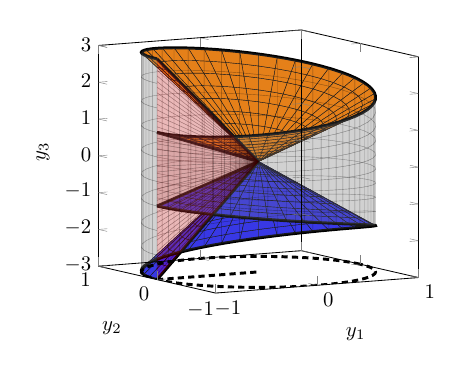}
    \caption{The sections $\mathbb{D}_{\text{s}} \odot \{-\widetilde{u}\}$ (blue) and $\mathbb{D}_{\text{s}} \odot \{\widetilde{u}\}$ (orange), their boundaries (thick black), and the the sector $\mathbb{D}_{\text{s}} \odot (-\widetilde{u},\widetilde{u})$ (gray) of Proposition \ref{prop:conj722}. Note that the sector in red does not meet $\mathbb{D}_{\text{s}} \odot (-\widetilde{u},\widetilde{u})$ but is a subset of its closure.}
    \label{fig:conj722}
\end{figure}
\begin{acknowledgements}\label{ackref}
The author would like
to thank Pierre Mathonet for numerous helpful suggestions and Naïm Zénaïdi for fruitful discussions.
\end{acknowledgements}

\bibliography{references.bib}

\affiliationone{
   Lucas Michel\\
   \mbox{Department of Mathematics, University of Liège,}\\
   \mbox{Allée de la découverte 12, 4000 Liège, Belgium}
   \email{lucas.michel@uliege.be}}
\end{document}